\documentclass[12pt,fleqn]{article}
\usepackage{pdfsync}
\usepackage{csquotes}
\usepackage[english]{babel}
\usepackage{mathtools,amsfonts,amsmath,amsthm,amssymb}
\def\bp{\begin{proof}}
\def\ep{\end{proof}}

\usepackage{color}

\definecolor{gr}{rgb}   {0.,   0.8,   0. } 
\definecolor{bl}{rgb}   {0.,   0.5,   1. } 
\definecolor{mg}{rgb}   {0.7,  0.,    0.7} 

\def\XXint#1#2#3{{\setbox0=\hbox{$#1{#2#3}{\int}$}
     \vcenter{\hbox{$#2#3$}}\kern-.5\wd0}}

 \usepackage[utf8]{inputenc}
\usepackage[T1]{fontenc}
 \usepackage{fouriernc}
\usepackage{mathrsfs} 
\usepackage{stmaryrd}
\usepackage{bbding}
 \usepackage{latexsym}
\usepackage{esint}

\usepackage[margin=17mm]{geometry}
 \numberwithin{equation}{section}

\usepackage{color}
 \definecolor{db}{rgb}{0.0,0.0,0.8} 
\definecolor{dg}{rgb}{0.0,0.55,0.14}
\definecolor{dr}{rgb}{0.5,0,0.07}

  \usepackage{hyperref}
 \hypersetup{frenchlinks=true,colorlinks=true,linkcolor=db,citecolor=dg,
 bookmarks=true}

 \usepackage{enumerate}

\theoremstyle{definition}

 \swapnumbers\newtheorem{theorem}{Theorem}[section]
 
\newtheorem{proposition}[theorem]{Proposition}
 
\newtheorem{lemma}[theorem]{Lemma}
\newtheorem{corollary}[theorem]{Corollary}
\theoremstyle{definition}

\theoremstyle{definition}
 
 \theoremstyle{definition}

\theoremstyle{definition}

\theoremstyle{definition}

\theoremstyle{definition}
\newtheorem{remark}[theorem]{Remark}
\theoremstyle{definition}

\newcounter{step}

\def\be{\begin{equation}}
\def\ee{\end{equation}}
\def\bes{\begin{equation*}}
\def\ees{\end{equation*}}
\def\bt{\begin{theorem}}
\def\et{\end{theorem}}
\def\bpr{\begin{proposition}}
\def\epr{\end{proposition}}
\def\bl{\begin{lemma}}
\def\el{\end{lemma}}
\def\bc{\begin{corollary}}
\def\ec{\end{corollary}}
\def\br{\begin{remark}}
\def\er{\end{remark}}
\def\ben{\begin{enumerate}}
\def\bena{\begin{enumerate}[a)]}
\def\een{\end{enumerate}}
\def\bit{\begin{itemize}}
\def\iit{\end{itemize}}

\def\supp{\operatorname{supp}}

\def\esssup{\operatorname{esssup}}

\DeclareMathAlphabet{\mathonebb}{U}{bbold}{m}{n}
\newcommand{\one}{\ensuremath{\mathonebb{1}}}
\def\R{{\mathbb R}}
\def\N{{\mathbb N}}

\def\Z{{\mathbb Z}}

\def\fo{\forall\, }

\def\va{\varphi}
\def\d{\displaystyle}

\def\ve{\varepsilon}

\def\l{\label}

\def\sr{{\cal S}}

\date{\today}
\title{Characterization of function spaces via low regularity mollifiers}
\author{
Xavier Lamy
\thanks{Universit\'e de Lyon,  CNRS UMR 5208, Universit\'e Lyon 1, Institut Camille Jordan, 43 blvd. du 11 novembre 1918, F-69622 Villeurbanne cedex, France. Email address: xlamy$@$math.univ-lyon1.fr}
\and
Petru Mironescu    \thanks{Universit\'e de Lyon,  CNRS UMR 5208, Universit\'e Lyon 1, Institut Camille Jordan, 43 blvd. du 11 novembre 1918, F-69622 Villeurbanne cedex, France. Email address: mironescu$@$math.univ-lyon1.fr}        
}

\begin{document}
\maketitle
\begin{abstract}{
Smoothness of a function $f:\R^n\to\R$ can be measured in terms of the rate of convergence of $f\ast\rho_\ve$ to $f$, where $\rho$ is an appropriate mollifier. In the framework of fractional Sobolev spaces, we characterize the \enquote{appropriate} mollifiers. We also obtain sufficient conditions, close to being necessary, which ensure that $\rho$ is adapted to a given scale of spaces. Finally, we examine in detail the case where $\rho$ is a characteristic function.}
\end{abstract}
\tableofcontents
\section{Introduction}
\l{s1}

${}$

The smoothness of a function $f:\R^n\to\R$ can be measured by different decay properties, for example via the decay properties of its harmonic extension, or the ones of its Littlewood-Paley decomposition,  or the ones of its coefficients in an appropriate wavelets frame. See \cite[Chapter~2]{triebel2} for a thorough discussion on this subject. Another characterization is related to the rate of convergence of $f\ast\rho_\ve$ to $f$, where $\rho$ is an appropriate mollifier. For example, for non integer $s>0$ and $1\le p<\infty$ we have 
\be
\l{aa1}
\|f\|_{W^{s,p}}^p\sim \|f\|_{L^p}^p+\int_0^1\frac 1{\ve^{sp+1}}\|f-f\ast\rho_\ve\|_{L^p}^p\, d\ve, \ \text{where }\rho_\ve(x)=\frac 1{\ve^n}\rho\left(\frac x\ve\right), 
\ee
provided
\be
\l{aa2}
\rho\in\sr\text{ and }\int\rho=1.
\ee
Here $\sr$ denotes the Schwartz class of smooth, rapidly decreasing functions.

We address here the question of the validity of \eqref{aa1} under assumptions as weak as possible on $\rho$. This is a \enquote{continuous} (vs \enquote{discrete}) counterpart of the analysis of Bourdaud \cite{bourdaud} concerning  the minimal assumptions required on the (father and mother) wavelets appropriate for the characterization of Besov spaces. 

Usually, the assumption $\rho\in\sr$ is weakened as follows. First, validity of \eqref{aa1} is established for some $\widetilde\rho\in\sr$. Next, one expresses an arbitrary $\rho$ in the form
\be
\l{ab1}
\rho=\sum_{j\ge 0}\eta^j\ast\widetilde\rho_{2^{-j}}\ \text{\cite[Lemma 2, p. 93]{stein}}.
\ee 
Then, using \eqref{ab1} and the validity of \eqref{aa2} for $\widetilde\rho$, it follows that property \eqref{aa1} holds for $\rho$ provided the $\eta^j$'s decay sufficiently fast. Finally, decay of $\eta^j$ is obtained by requiring a sufficient decay of the Fourier transform $\widehat\rho$ of $\rho$. With more work, spatial conditions on $\rho$ (of Fourier multiplier's theorem type) ensure the decay of $\widehat\rho$ and thus lead to (usually suboptimal) sufficient conditions for the validity of \eqref{aa1}.\footnote{A typical result for which this approach is followed is the fact that the norm on the Besov spaces $B^s_{p,q}$ does not depend on the choice of the rapidly decreasing mollifier; see \cite[Section 2.3, p. 168]{triebel1} and the use of the Fourier multipliers theory \cite[Section 2.2.4, p. 161]{triebel1}.} Alternatively, in standard function spaces one can rely on the decomposition of functions in simple building blocks (e.g. atoms) and obtain almost sharp spatial sufficient conditions.  For such an approach in the framework of the Hardy spaces, see  \cite{steintw}, \cite{steinnon}.

In what follows, we will obtain, using very little technology,  necessary and sufficient conditions on $\rho$ in order to have \eqref{aa1}, and simple sufficient spatial conditions on $\rho$, close to being optimal.

Of special interest to us will be the validity of \eqref{aa1} when $f\ast\rho_\ve$ is particularly simple to compute. A typical example consists in taking   $\rho$ the characteristic function of a unit cube, e.g. $Q=(0,1)^n$ or $Q=(-1/2,1/2)^n$. We will determine the  spaces $W^{s,p}$ which can be described via such a $\rho$. 

It turns out that our techniques are adapted not only to the Sobolev spaces with non integer $s$, but more generally to the Besov spaces $B^s_{p,q}$ with $s>0$, $1\le p\le\infty$ and $1\le q\le \infty$. Recall that this scale of spaces includes the one of fractional Sobolev spaces, since $W^{s,p}=B^s_{p,p}$ for non integer $s$ \cite[Chapter~2]{triebel2}. For simplicity, we will write all our formulas and statements only when $q<\infty$. However, our results hold also when $q=\infty$, and the corresponding results are obtained by straightforward adaptations of the formulas and arguments.

Our first result is a one sided estimate, which surprisingly requires no smoothness of $\rho$.
\bt
\l{aa3}
Let $\rho\in L^1$ be such that $\int\rho=1$. Then we have
\be
\l{aa4}
\|f\|_{B^{s}_{p,q}}^q\lesssim \|f\|_{L^p}^q+\int_0^1\frac 1{\ve^{sq+1}}\|f-f\ast\rho_\ve\|_{L^p}^q\, d\ve, \ \fo s>0, \fo 1\le p\le\infty,\ \fo 1\le q<\infty.
\ee
\et 
\br
It is tempting to extend Theorem \ref{aa3} to finite measures, but the example $\rho=\delta_0$ (the Dirac mass at the origin) shows that Theorem \ref{aa3} need not hold for a measure. We do not know how to characterize the finite measures of total measure $1$ satisfying \eqref{aa4}.
\er

We next discuss what is needed in order to obtain the reverse of \eqref{aa4}. For this purpose, we fix some  $\eta\in\sr$. Assuming that the reverse of \eqref{aa4} holds, we have
\be
\l{aa5}
\int_0^1\frac 1{\ve^{sq+1}}\|\eta-\eta\ast\rho_\ve\|_{L^p}^q\, d\ve<\infty.
\ee 

It turns out that \eqref{aa5} with $p=q=1$ is also sufficient. 
\bt
\l{ab2}
Let $\rho\in L^1$ satisfy $\int\rho=1$. Let $s>0$. Then the following are equivalent.
\ben
\item
There exists some $\eta\in\sr$ such that $\int\eta\neq 0$ and
\be
\l{ac1}
\int_0^1\frac 1{\ve^{s+1}}\|\eta-\eta\ast\rho_\ve\|_{L^1}\, d\ve<\infty.
\ee
\item
For every $1\le p\le\infty$ and every $1\le q<\infty$ we have
\be
\l{ab3}
\|f\|_{B^{s}_{p,q}}^q\sim \|f\|_{L^p}^q+\int_0^1\frac 1{\ve^{sq+1}}\|f-f\ast\rho_\ve\|_{L^p}^q\, d\ve.
\ee
\een
\et
An additional equivalent characterization of $\rho$ satisfying the above properties will be provided in Section \ref{s4}.

We now turn to the case where $\rho$ is the characteristic function of a set $A$. In that case, the range of values of $s$ for which the equivalent characterizations of Theorem~\ref{ab2} are satisfied depends only on whether or not the set $A$ is centered:

\bpr
\l{ab5}
Let $\rho=\d\frac{1}{|A|}\one_A$, where $A\subset\R^n$ is a bounded measurable set of positive Lebesgue measure. Then $\rho$ characterizes all the spaces $B^s_{p,q}$ for fixed $s$ (that is, \eqref{ab3} is valid) if and only if:
\ben
\item
Either $\int_A y\, dy =0$ and $s<2$.
\item
Or $\int_A y\, dy \neq 0$ and $s<1$.
\een
\epr

Finally, we provide sufficient spatial conditions for the validity of \eqref{ab3} when $0<s<1$.
\bpr
\l{ab7}
Let $\rho\in L^1$ satisfy $\int\rho =1$, and $0<s<1$. If $\rho$ satisfies the moment condition
\begin{equation}\label{moment}
\int |y|^s |\rho(y)|\, dy <\infty,
\end{equation} 
then $\rho$ characterizes all spaces $B^s_{p,q}$. That is, \eqref{ab3} is valid.
\epr

For $s\geq 1$, the exemple of $\rho=\one_A$ with uncentered $A$ shows that there is no such simple sufficient finite moment condition. In order to obtain the validity of \eqref{ab3} for higher $s$, one would need to ask for the vanishing of moments, as in the case of $\rho=\one_A$. For more details see Proposition~\ref{pmoments} below.

The sufficient spatial condition \eqref{moment} turns out to be optimal, in the sense that for non negative $\rho$ it is also necessary:
\bpr
\l{ab8}
Let $s>0$. Let $\rho\in L^1$ satisfy $\int\rho=1$ and $\rho\geq 0$. If \eqref{ab3} is valid, then $\rho$ necessarily satisfies the moment condition \eqref{moment}.
\epr

The plan of the paper is as follows. In Section~\ref{s2} we introduce some preliminary notation, definitions and tools required in the sequel. In Sections~\ref{s3} and \ref{s4} we prove our two main results, Theorems \ref{aa3} and \ref{ab2}. Eventually, Section~\ref{s5} is devoted to proving Propositions \ref{ab5}, \ref{ab7} and \ref{ab8}.

\subsubsection*{Acknowledgments}

Part of this work was carried out while XL was visiting McMaster University. He thanks the Mathematics and Statistics department of McMaster University, and in particular S.~Alama and L.~Bronsard, for their hospitality.
PM was partially supported by the ANR project \enquote{Harmonic Analysis at its Boundaries},   ANR-12-BS01-0013-03.
\section{Preliminaries}
\l{s2}

\subsection{Littlewood-Paley decomposition and $B^s_{p,q}$}\label{s2s1}

We will make use of the (inhomogeneous) Littlewood-Paley decomposition of a temperate distribution.
Let $\zeta,\varphi\in\mathcal S(\mathbb R^n)$ be as follows:
\begin{itemize}
\item
$\supp \widehat \zeta \subset B(0,2)$ and $\widehat\zeta\equiv 1$ in a neighborhood of $\overline B (0,1)$,
\item $\varphi:=\zeta_{1/2}-\zeta$, so that $\widehat\varphi =\widehat\zeta(\cdot/2)  -\widehat\zeta$ and $\supp\widehat\varphi\subset B(0,4)\setminus\overline B(0,1)$.
\end{itemize}
The (inhomogeneous) Littlewood-Paley decomposition of a temperate distribution $f\in\mathcal S'(\mathbb R^n)$ is then given by
\begin{equation}\label{LP}
f=\sum_{j\geq 0} f_j,\quad\text{where }f_0=f\ast\zeta\text{ and }f_j=f\ast \varphi_{2^{1-j}}\text{ for }j\geq 1.
\end{equation}
See for instance \cite[Section~VI.4.1]{stein}.

The Littlewood-Paley  decomposition can be used to characterize the space $B^s_{p,q}$  \cite[Section 2.3.2, Proposition 1, p. 46]{triebel2}, and this is the definition we adopt here:
\begin{equation}\label{seminorm}
B^s_{p,q}=\left\lbrace f\in L^ p;\,   |f|_{B^s_{p,q}}^q := \sum_{j\geq 0} 2^{sjq}\|f_j\|_{L^p}^q <\infty \right\rbrace.
\end{equation}
The norm on $B^s_{p,q}$ is defined by
\begin{equation}\label{norm}
\| f\|^q_{B^s_{p,q}} = \|f\|_{L^p}^q + |f|_{B^s_{p,q}}^q.
\end{equation}
Different choices of $\zeta$ yield equivalent norms \cite[Section~2.3]{triebel3}. See also \cite[Chapter~3]{triebel3} for other equivalent characterizations of $B^s_{p,q}$.

\subsection{Schur's criterion}\label{s2s2}

We will also make use of the following Schur-type estimate for kernel operators; see e.g. \cite[Appendix I]{grafakosclassical}.

\begin{lemma}\label{schur}
Let $(X,\mu)$ and $(Y,\nu)$ be two ($\sigma$-finite) measure spaces, let $1\leq p\leq\infty$, and $\kappa\colon X\times Y \to\mathbb C$ a measurable kernel. If the quantities
\begin{equation*}
M_1:=\esssup_x \int |\kappa(x,y)|\, d\nu(y)\quad\text{and}\quad M_2:=\esssup_y\int |\kappa(x,y)|\, d\mu(x),
\end{equation*}
are finite, then the formula
\begin{equation*}
Tu(x)=\int \kappa(x,y)u(y)\,  d\nu(y)
\end{equation*}
defines a bounded linear operator from $L^p(Y)$ to $L^p(X)$, with norm
\begin{equation*}
\| T\| \leq M_1^{1/p'}M_2^{1/p}.
\end{equation*}
Here $p'=p/(p-1)$ is the conjugate exponent of $p$.
\end{lemma}

\section{Proof of Theorem \ref{aa3}}
\l{s3}

\bp[Proof of Theorem \ref{aa3}]
We are going to prove a discrete version of \eqref{aa4}. We start from the identity
\begin{equation}\label{discr}
\int_0^1 \|f-f\ast\rho_\ve\|_{L^p}^q\, \frac{d\ve}{\ve^{sq+1}} = \sum_{j\geq 0} 2^{sjq}\int_{1/2}^1\|f-f\ast\rho_{2^{-j}\ve}\|_{L^p}^q\, d\ve.
\end{equation}
In view of \eqref{discr}, it suffices to establish the estimate
\begin{equation}\label{aa3toshow}
\|f\|^q_{B^s_{p,q}} \lesssim \|f\|_{L^p}^q + \sum_{j\geq 0} 2^{sjq}\|f-f\ast\rho_{2^{-j}\ve}\|_{L^p}^q,
\end{equation}
uniformly with respect to $\ve\in (1/2,1)$. Integrating \eqref{aa3toshow} and using \eqref{discr}, we obtain indeed the desired inequality \eqref{aa4}.

To simplify the notation, we will establish \eqref{aa3toshow} for $\ve=1$, which amounts to considering $\widetilde\rho=\rho_\ve$ instead of $\rho$. It will be clear at the end of the proof that all estimates are indeed uniform with respect to $\ve\in (1/2,1)$.

We introduce a function $\psi\in\sr$ satisfying the following:
\begin{equation}\label{psi}
\widehat\psi\equiv 1\text{ on }\mathrm{supp}\:\widehat\varphi,\ \text{ and }\widehat\psi (0)=0.
\end{equation}
Recall that $\varphi$ is the function used in the definition of the Littlewood-Paley decomposition \eqref{LP}. Since the support of $\widehat\varphi$ is contained in the annulus $\lbrace 1\leq |\xi |\leq 4 \rbrace$, it is indeed possible to choose $\psi$ satisfying \eqref{psi}.

We need to estimate the $B^s_{p,q}$ semi-norm of $f$, hence the sum
\begin{equation*}
\sum_{j\geq 0} 2^{sjq}\|f_j\|_{L^p}^q,
\end{equation*}
where $f=\sum_j f_j$ is the Littlewood-Paley decomposition \eqref{LP}. We introduce an integer $k>0$, to be fixed later, and split the sum into two parts:
\begin{equation}\label{aa3proof1}
|f|_{B^s_{p,q}}^q\leq \sum_{ j \leq k} 2^{sjq}\|f_j\|_{L^p}^q + \sum_{j> k} 2^{sjq}\|f_j\|_{L^p}^q.
\end{equation}
Using the fact that
\begin{equation*}
\|f_j\|_{L^p} = \| f\ast\varphi_{2^{1-j}} \|_{L^p} \leq \|f\|_{L^p}\|\varphi\|_{L^1},\ \fo j\ge 1,\ \text{ and } \|f_0\|_{L^p}=\|f\ast\zeta\|_{L^p}\le \|f\|_{L^p}\|\zeta\|_{L^1}, 
\end{equation*}
we simply estimate the first sum in the right-hand side of \eqref{aa3proof1} by
\begin{equation}\label{aa3proof2}
\sum_{j \leq k} 2^{sjq}\|f_j\|_{L^p}^q \lesssim \|f\|_{L^p}^q.
\end{equation}

We  next turn to estimating the second sum. We will use the notation 
$\rho^j :=\rho_{2^{-j}}$,
and similarly $\varphi^j$, and so on.

Taking advantage of the fact that $\psi\ast\varphi=\varphi$ (and thus $\psi^j\ast\varphi^j=\varphi^j$)
 we write, for $j> k$,
\begin{equation*}
\begin{aligned}
f_{j+1} & = (f-f\ast\rho^{j-k}+f\ast\rho^{j-k})\ast\varphi^j \\
& =(f-f\ast\rho^{j-k})\ast\varphi^j+f\ast\rho^{j-k}\ast\psi^j\ast\varphi^j \\
& = (f-f\ast\rho^{j-k})\ast\varphi^j + f_{j+1}\ast(\rho\ast\psi^k)^{j-k}.
\end{aligned}\end{equation*}
We deduce the estimate
\begin{equation}\label{aa3proof3}
\|f_{j+1}\|_{L^p} \leq \|\varphi\|_{L^1}\|f-f\ast\rho^{j-k}\|_{L^p} + \|\rho\ast\psi^k\|_{L^1}\|f_{j+1}\|_{L^p}.
\end{equation}
Since $\widehat\psi(0)=0$, we can apply Lemma~\ref{keylem} below: it holds
\begin{equation}\label{aa3proof4}
\|\rho\ast\psi^k\|_{L^1} = \|\rho\ast\psi_{2^{-k}}\|_{L^1} \rightarrow 0,\quad\text{as }k\to \infty.
\end{equation}
Thus for sufficiently large $k$ we may absorb the last term of the right-hand side of \eqref{aa3proof3} into the left-hand side. For such $k$, we have
\begin{equation}\label{aa3proof5}
\|f_{j+1}\|_{L^p} \lesssim \|f-f\ast\rho^{j-k}\|_{L^p}\quad\text{for }j\geq k.
\end{equation}
Plugging \eqref{aa3proof5} into \eqref{aa3proof1} and recalling \eqref{aa3proof2}, we obtain
\begin{equation}\label{aa3proof6}
\|f\|^q_{B^s_{p,q}} \lesssim \|f\|_{L^p}^q + \sum_{j\geq 0} 2^{sjq}\|f-f\ast\rho_{2^{-j}}\|_{L^p}^q.
\end{equation}
The latter estimate is exactly the desired estimate \eqref{aa3toshow} with $\varepsilon=1$. The corresponding estimate for $1/2\leq \varepsilon\leq 1$ is found by replacing $\rho$ with $\widetilde\rho=\rho_\ve$ in the proof of \eqref{aa3proof6}. The resulting estimate is uniform with respect to $\ve\in (1/2,1)$ thanks to Remark~\ref{keyrem} following Lemma~\ref{keylem}. This concludes the proof of Theorem~\ref{aa3}.
\ep

\begin{lemma}\label{keylem}
Let $\rho\in L^1$, and let $\psi\in L^1$ satisfy $\int\psi=0$. Then
\begin{equation*}
\lim_{\ve\to 0} \| \rho\ast\psi_\ve\|_{L^1} = 0.
\end{equation*}
\end{lemma}

\begin{remark}\label{keyrem}
If we apply Lemma~\ref{keylem} to $\rho_\delta$ instead of $\rho$, with $1/2\leq \delta\leq 1$, then the resulting convergence is uniform with respect to $\delta$, as is shown by the following computation:
\begin{equation*}
\lim_{\ve\to 0}\, \sup_{1/2\leq \delta\leq 1} \, \|\rho_\delta\ast\psi_\ve\|_{L^1} = \lim_{\ve\to 0} \, \sup_{1/2\leq \delta\leq 1} \, \|(\rho\ast\psi_{\ve/\delta})_\delta\|_{L^1}   = \lim_{\ve\to 0} \, \sup_{1/2\leq \delta\leq 1}\,  \|\rho\ast\psi_{\ve/\delta}\|_{L^1} = \lim_{\ve\to 0} \| \rho\ast\psi_\ve\|_{L^1}.
\end{equation*}
\end{remark}

\begin{proof}[Proof of Lemma~\ref{keylem}]
We introduce a parameter $R>0$.
Taking advantage of the fact that $\int\psi=0$, we may write
\begin{equation}\label{key1}
\begin{aligned}
\rho\ast\psi_\varepsilon (x) & = 
\frac{1}{\varepsilon^n}\int \Big(\rho(y)-\fint_{B_{R\ve}(x)}\rho \Big) \psi\left( \frac{x-y}{\ve}\right) \, dy \\
& = \frac{1}{R^n\varepsilon^{2n}\omega_n} \iint_{|z-x|<R\ve} (\rho(y)-\rho(z))\psi\left(\frac{x-y}{\ve}\right)\, dydz,
\end{aligned}\end{equation}
where $B_{R\ve}(x)$ is the open ball of center $x$ and radius $R\ve$, and $\omega_n$ is the Lebesgue measure of the unit ball. We then have
\begin{equation}\label{key2}
\int |\rho\ast\psi_\ve(x)|\, dx \leq \int A_R(x)\, dx + \int B_R(x)\, dx,
\end{equation}
where
\begin{align}
A_R(x) & = \frac{1}{R^n\ve^{2n}\omega_n}\iint_{\substack{|z-x|<R\varepsilon \\ |y-x|<R\ve}}|\rho(y)-\rho(z)|\left|\psi\left(\frac{x-y}{\ve}\right)\right|\, dydz,\label{key3}\\
B_R(x) & = \frac{1}{R^n\ve^{2n}\omega_n}\iint_{\substack{|z-x|<R\varepsilon \\ |y-x|\geq R\ve}}(|\rho(y)|+|\rho(z)|)\left|\psi\left(\frac{x-y}{\ve}\right)\right|\, dydz. \label{key4}
\end{align}

To estimate $\int A_R(x)\, dx$, we perform the change of variable $x\leadsto w=(x-y)/\ve$ and find
\begin{align*}
\int A_R(x)\, dx & \le \frac{1}{R^n\varepsilon^n\omega_n} \int_{|w|<R}|\psi(w)|\,dw \iint_{|z-y|<2R\ve}|\rho(y)-\rho(z)|\, dy dz \\
& \leq \frac{\|\psi\|_{L^1}}{R^n\ve^n\omega_n}\int_{|h|<2R\ve} \|\rho(\cdot+h)-\rho\|_{L^1} \, dh,
\end{align*}
and thus
\begin{equation}\label{key5}
\int A_R(x)\, dx \leq 2^n\|\psi\|_{L^1}\sup_{|h|<2R\ve}\|\rho(\cdot+h)-\rho\|_{L^1}.
\end{equation}
Note that, for any fixed $R$, the right-hand side of \eqref{key5} converges to 0 as $\ve\to 0$.

We next estimate $\int B_R(x)\, dx$. To this end we compute
\begin{equation}\label{key6}
\begin{aligned}
\frac{1}{R^n\ve^{2n}\omega_n}\iiint_{\substack{|z-x|<R\ve \\ |y-x|\geq R\ve}} |\rho(y)|\left|\psi\left(\frac{x-y}{\ve}\right)\right| \, dxdydz & 
=\frac{1}{\ve^n}\iint_{|y-x|\geq R\ve} |\rho(y)| \left|\psi\left(\frac{x-y}{\ve}\right)\right| dx dy \\
& = \|\rho\|_{L^1}\int_{|w|\geq R}|\psi(w)|\, dw,
\end{aligned}\end{equation}
and
\begin{equation}\label{key7}
\begin{aligned}
\frac{1}{R^n\ve^{2n}\omega_n}\iiint_{\substack{|z-x|<R\ve \\ |y-x|\geq R\ve}} |\rho(z)|\left|\psi\left(\frac{x-y}{\ve}\right)\right| \, dxdydz & 
=\frac{1}{R^n\ve^n\omega_n}\int_{|w|\geq R}\!\!\! |\psi(w)|\,dw \iint_{|z-x|< R\ve}\!\!\!\! |\rho(z)|  dx dz \\
& = \|\rho\|_{L^1}\int_{|w|\geq R}|\psi(w)|\, dw.
\end{aligned}\end{equation}
Plugging \eqref{key6} and \eqref{key7} into formula \eqref{key4}, we obtain
\begin{equation}\label{key8}
\int B_R(x) \, dx \le C \|\rho\|_{L^1} \int_{|w|\geq R}|\psi(w)|\, dw.
\end{equation}
Combining \eqref{key2}, \eqref{key5} and \eqref{key8} we obtain
\begin{equation*}
\limsup_{\ve\to 0} \|\rho\ast\psi_\ve\|_{L^1}\leq C \|\rho\|_{L^1} \int_{|w|\geq R}|\psi(w)|\, dw,
\end{equation*}
and complete  the proof of Lemma~\ref{keylem} by letting $R\to \infty$.
\end{proof}

\section{Proof of Theorem \ref{ab2}}
\l{s4}

${}$
\bp[Proof of Theorem \ref{ab2}]
We clearly have \enquote{$2\implies 1$}, and it remains to prove that \enquote{$1\implies 2$}. For the convenience of the reader, we start by establishing a consequence of property 1, and then we proceed to the proof of the desired implication.

\medskip
\noindent
{\it Step 1.} A discrete-uniform version of 1.\\
Assume that property 1 holds. Then we claim that for every $\va\in\sr$ we have
\be
\l{vw3}
\sup_{1/2\le\ve\le 1}\ \sum_{j\ge 0}2^{sj}\|\va-\va\ast\rho_{2^{-j}\ve}  \|_{L^1}\le C<\infty.
\ee
In order to prove \eqref{vw3}, we start from the following fact. We fix a function $\lambda\in\sr$ such that $\int\lambda\neq 0$. Then every function $\psi\in{\mathscr S}(\R^n)$ may be written as 
\be
\l{baf1}
\psi=\sum_{k\ge 0}\lambda^{k}_\psi\ast\lambda_{2^{-k}}.
\ee
Here $(\lambda^{k}_\psi)_k\subset  \sr$ is a sequence that decays rapidly as $k\to\infty$, in the following sense: if $\psi$ belongs to a bounded subset ${\mathscr B}\subset \sr$, then for every $M>0$  there  exists a constant $C$ such that
\be
\l{baf2}
\|\lambda^{k}_\psi\|_{L^1}\le \frac{C}{2^{Mk}},\ \fo k\ge 0,\ \fo\psi\in {\mathscr B};
\ee
see \cite[Lemma 2, p. 93]{stein}. In particular, if we fix $\va\in\sr$ then we may write
\be
\l{aux1}
\va_t=\sum_{k\ge 0}\lambda^{k,t}\ast\lambda_{2^{-k}}, \ \fo t\in [1,2],
\ee
with
\be
\l{aux2}
\|\lambda^{k,t}\|_{L^1}\le \frac{C}{2^{Mk}},\ \fo k\ge 0,\ \fo t\in [1,2].
\ee
We now choose an appropriate $\lambda\in\sr$. In view of \eqref{discr}, if property 1 holds then we may find some $\ve\in [1/2,1]$ such that $\lambda:=\eta_{1/\ve}$ satisfies
\be
\l{aux3}
\sum_{k\ge 0}2^{sk}\|\lambda-\lambda\ast\rho_{2^{-k}}\|_{L^1}=\sum_{k\ge 0}2^{sk}\|\eta-\eta\ast\rho_{2^{-k}\ve}\|_{L^1}<\infty.
\ee
By combining \eqref{aux1}-\eqref{aux3} we find that, with $\ve\in [1/2,1]$ and $t:=1/\ve\in  [1,2]$, we have
\bes
\begin{aligned}
\sum_{j\ge 0}2^{sj}\|\va-\va\ast\rho_{2^{-j}\ve}\|_{L^1}&=\sum_{j\ge 0}2^{sj}\|\va_t-\va_t\ast\rho_{2^{-j}}\|_{L^1}\le \sum_{j\ge 0}2^{sj}\sum_{k\ge 0}\|\lambda^{k,t}\ast\lambda_{2^{-k}}-\lambda^{k,t}\ast\lambda_{2^{-k}}\ast\rho_{2^{-j}}\|_{L^1}\\
&\le \sum_{j\ge 0}\sum_{k\ge 0}2^{sj}\|\lambda^{k,t}\|_{L^1}\|\lambda_{2^{-k}}-\lambda_{2^{-k}}\ast\rho_{2^{-j}}\|_{L^1}\\
&\le C\sum_{j\ge 0}\sum_{k>j}2^{sj}\|\lambda^{k,t}\|_{L^1}+\sum_{j\ge 0}\sum_{k\le j}2^{sj}\|\lambda^{k,t}\|_{L^1}\|\lambda-\lambda\ast\rho_{2^{k-j}}\|_{L^1}\\
&\le C\sum_{j\ge 0}\sum_{k>j}2^{sj}2^{-(s+1)k}+\sum_{\ell\ge 0}\sum_{j\ge \ell}2^{sj}\|\lambda^{j-\ell,t}\|_{L^1}\|\lambda-\lambda\ast\rho_{2^{-\ell}}\|_{L^1}\\
&\le C+C\sum_{\ell\ge 0}\sum_{j\ge \ell}2^{sj}2^{-(s+1)(j-\ell)}\|\lambda-\lambda\ast\rho_{2^{-\ell}}\|_{L^1}\\
&\le C+C\sum_{\ell\ge 0}2^{s\ell}\|\lambda-\lambda\ast\rho_{2^{-\ell}}\|_{L^1}\le C,
\end{aligned}
\ees
with constants independent of $t$, 
i.e., \eqref{vw3} holds.

\medskip
\noindent{\it Step 2.} Proof of \enquote{$1\implies 2$}.\\
As we proved in the previous step, we may assume that  there exists some $\eta\in\sr$ such that 
\be
\l{tbb}
\widehat\eta\equiv 1\text{ in }B(0,4),
\ee
and such that $\eta$ satisfies the following uniform and discrete version of \eqref{ac1}:
\be
\l{ac2}
S_\ve:=\sum_{j\ge 0}2^{sj}\|\eta-\eta\ast\rho_{2^{-j}\ve}\|_{L^1}\le C,\ \fo\ve\in [1/2,1],\  \text{with }C\text{ independent of }\ve\in [1/2,1].
\ee

Let $f\in L^p$. We will establish the estimate
\be
\l{ac3}
\sum_{j\ge 0}2^{sjq}\|f-f\ast\rho_{2^{-j}\ve}\|_{L^p}^q\le  C\ (1+S_\ve)^q\  |f|_{B^s_{p,q}}^q, \fo\ve\in [1/2,1], \ \text{with }C\text{ independent of }\ve\in [1/2,1].
\ee

We obtain \eqref{ab3} by integrating \eqref{ac3} in $\ve$ and using \eqref{ac2}.

In turn, estimate \eqref{ac3} is obtained as follows. Set
\be
\l{ad5}
 \alpha_{j,\ve}:=2^{sj}\|\eta-\eta\ast\rho_{2^{-j}\ve}\|_{L^1}, \text{ which satisfies }\sum_{j\ge 0}\alpha_{j,\ve }\le C, \fo \ve\in [1/2,1].
\ee

Let $f=\sum_{\ell\ge 0}f_\ell$ be the (inhomogeneous) Littlewood-Paley decomposition of  $f\in L^p$, defined in Section~\ref{s2s1}. By \eqref{tbb}, for every $\ell$ we have $f_\ell=f_\ell\ast\eta_{2^{-\ell}}$, and thus 
\be
\l{ad1}
\begin{aligned}
f-f\ast\rho_{2^{-j}\ve}&=\sum_{\ell\ge 0}(f_\ell-f_\ell\ast\rho_{2^{-j}\ve})=\sum_{\ell\ge j}(f_\ell-f_\ell\ast\rho_{2^{-j}\ve})+\sum_{\ell< j}(f_\ell-f_\ell\ast\rho_{2^{-j}\ve})\\
&=\sum_{\ell\ge j}(f_\ell-f_\ell\ast\rho_{2^{-j}\ve})+\sum_{\ell< j}f_\ell\ast(\eta_{2^{-\ell}}-\eta_{2^{-\ell}}\ast\rho_{2^{-j}\ve})\\
&=\sum_{\ell\ge j}(f_\ell-f_\ell\ast\rho_{2^{-j}\ve})+\sum_{\ell< j}f_\ell\ast(\eta-\eta\ast\rho_{2^{\ell-j}\ve})_{2^{-\ell}}.\end{aligned}
\ee

Using \eqref{ad1}, we find that
\be
\l{ad2}
\|f-f\ast\rho_{2^{-j}\ve}\|_{L^p}\lesssim \sum_{\ell\ge j}\|f_\ell\|_{L^p}+
\sum_{\ell< j}2^{-s(j-\ell)}\alpha_{j-\ell,\ve}\|f_\ell\|_{L^p},
\ee
i.e.,
\be
\label{ad3}
2^{sj}\|f-f\ast\rho_{2^{-j}\ve}\|_{L^p} \lesssim \sum_\ell \left[ 2^{s(j-\ell)}\one_{\{\ell\geq j\}}(\ell) + \alpha_{j-\ell,\ve}\one_{\{\ell<j\}}(\ell) \right] 2^{s\ell}\|f_\ell\|_{L^p}.
\ee

We obtain \eqref{ac3} by combining \eqref{ad5} with \eqref{ad3} and with Schur's criterion (Lemma~\ref{schur}) applied to $X=Y=\ell^q$ and $k(j,\ell)=2^{s(j-\ell)}\one_{\{\ell\geq j\}}(\ell) + \alpha_{j-\ell,\ve}\one_{\{\ell<j\}}(\ell)$.
\ep

We continue with another characterization of the kernels $\rho$ satisfying the equivalent properties 1 and 2 in Theorem \ref{ab2}. For simplicity, the main results of our article were stated for inhomogeneous Besov spaces. It turns out that the homogeneous version of our next result is easier to understand than the inhomogeneous one, so that we start by presenting (without proof) the homogeneous cousin of Theorem \ref{sv1} below. 

In order to avoid subtle issues concerning the realization of homogeneous Besov spaces as spaces of distributions, we consider only  temperate distributions $f$ such that 
\be
\l{vu2}
 \widehat f\text{ is compactly supported in }\R^n\setminus\{0\}.
\ee 
Any such $f$ is smooth, and we have $f=\sum_{j\in\Z}f_j$ in $\sr'$, where (in the spirit of \eqref{LP}) $f_j=f\ast \va_{2^{1-j}}$, $\fo j\in\Z$. For $f$ satisfying \eqref{vu2}, we set
\bes
|f|_{\dot B^s_{p,q}}^q=\sum_{j\in\Z}2^{sjq}\|f_j\|_{L^p}^q,
\ees
with the obvious modification when $q=\infty$. Let us note that, the series $\sum_{j\in\Z}f_j$ containing only a finite number of non zero terms, we actually have 
\bes
\dot B^s_{p,q}=\{ f\in L^p(\R^n);\, f\text{ satisfies }\eqref{vu2}\},
\ees
but that the norm we consider is not equivalent to the $L^p$ norm.

As in the inhomogeneous case considered in this article, we may try to characterize the $L^1$ kernels $\rho$ such that 
\be
\l{vu3}
|f|_{\dot B^s_{p,q}}^q\sim\int_0^\infty\frac 1{\ve^{sq+1}}\|f-f\ast\rho_\ve\|_{L^p}^q\, d\ve,\ \text{ for every }f\text{ satisfying }\eqref{vu2}.
\ee
The homogeneous counterpart of Theorem \ref{ab2} consists of the following equivalence: for a fixed $s$ (not necessarily positive) \eqref{vu3} holds if and only if for a function $\va$ as in the Littlewood-Paley decomposition we have 
\be
\l{vu4}
\int_0^\infty\frac 1{\ve^{s+1}}\|\va-\va\ast\rho_\ve\|_{L^1}\, d\ve<\infty.
\ee
Necessity of \eqref{vu4} comes from the fact that \eqref{vu3}  holds with $p=q=1$. 

Let us now examine what is required in order to have \eqref{vu3} when $p=q=\infty$. If \eqref{vu3} holds  and if $|f|_{\dot B^s_{\infty,\infty}}<\infty$, then  the distribution
\bes
f-f\ast\rho_\ve=(\delta-\rho)_\ve\ast f
\ees
is  well-defined (as the convolution of a finite measure with a smooth bounded function). Moreover, $\|f-f\ast\rho_\ve\|_{L^\infty}$ is   controlled by the norm $|f|_{\dot B^s_{\infty,\infty}}$ (since \eqref{vu3} holds). A  moment thought shows that in particular  $\delta-\rho$ is an element of the dual of $\dot B^s_{\infty, \infty}$. Remarkably, this necessary condition is also sufficient, and is equivalent to the property \eqref{vu4}.

Theorem \ref{sv1} is the inhomogeneous counterpart of the above fact. In order to state this result, it is convenient to define  ad hoc norm and  function space. Fix $\zeta$, $\va$ as in the Littlewood-Paley decomposition \eqref{LP}. In order to simplify the proof of Theorem \ref{sv1}, we make the (unessential) assumption that
\be
\l{abade}
\va\text{ is even}.
\ee

Our appropriate function space is defined starting from the identity
\be
\l{vv1}
f=(f-f\ast\zeta)+\sum_{j\le -1}f\ast\va_{2^{-j}} := \sum_{j\le 0}f^\sharp_j,\ \fo f\in\sr'\text{ satisfying }\eqref{vu2}.
\ee
We define the appropriate norm
\be
\l{vv2}
[f]_{X^s_{p,q}}^q=\sum_{j\le 0}2^{sjq}\|f^\sharp_j\|_{L^p}^q,
\ee
with the corresponding modification when $q=\infty$. Let $X^s_{p,q}$ be the space of temperate distributions satisfying \eqref{vu2} and such that $[f]_{X^s_{p,q}}<\infty$.\footnote{This space is $\{ f\in L^p(\R^n);\, f\text{ satisfies }\eqref{vu2}\}$, but not with the $L^p$ norm.}

\bt
\l{sv1}
Let $s>0$. Then
property \eqref{ac1} is equivalent to 
\be
\l{vw1}
\delta-\rho\in \left(X^s_{\infty, \infty}\right)^\ast.
\ee
\et
\bp

{}
${}$

\noindent
\enquote{\eqref{ac1}$\implies$\eqref{vw1}}.
Let $\va$ be as in the Littlewood-Paley decomposition and let $\psi$ be as in \eqref{psi}. We may assume that $\psi$ is even. If $f\in\sr'$ and $\ve>0$ are such that $f\ast\va_\ve\in L^\infty$, then we have
\bes
\begin{aligned}
 (\delta-\rho)(f\ast\va_\ve)&=  (\delta-\rho)(f\ast\va_\ve\ast\psi_\ve)=[(\delta-\rho)\ast\psi_\ve](f\ast\va_\ve)=\int \left[(\delta-\rho)\ast\psi_\ve (x)\right]\, \left[f\ast\va_\ve(x)\right]\, dx.
\end{aligned}
\ees
In particular, if $j<0$ and $f\in X^s_{\infty, \infty}$, then 
\be
\l{ua1}
\left|(\delta-\rho)(f_j^\sharp)\right|=\left|\int (\delta-\rho)\ast\psi_{2^{-j}}(x)\, f_j^\sharp(x)\, dx \right|\le \|(\delta-\rho)\ast\psi_{2^{-j}}\|_{L^1}\|f_j^\sharp\|_{L^\infty}.
\ee

On the other hand, for $j=0$ we have $f_0^\sharp\in C^\infty\cap L^\infty$ (in view of \eqref{vu2} and of the definition of $X^s_{\infty, \infty}$) and thus
\be
\l{ua2}
\left|(\delta-\rho)(f_0^\sharp)\right|\le (1+\|\rho\|_{L^1})\|f_0^\sharp\|_{L^\infty}.
\ee

We next note that \eqref{vw3} (applied to $\psi$ instead of $\va$), which is a consequence of \eqref{ac1}, implies that
\be
\l{ua3}
\begin{aligned}
 \sum_{j<0}2^{-sj}\|(\delta-\rho)\ast\psi_{2^{-j}}\|_{L^1}& =\sum_{j<0}2^{-sj}\|\psi_{2^{-j}}-\rho\ast\psi_{2^{-j}}\|_{L^1}=\sum_{j<0}2^{-sj}\|\psi-\psi\ast\rho_{2^j}\|_{L^1}\\
 &=\sum_{k>0}2^{sk}\|\psi-\psi\ast\rho_{2^{-k}}\|_{L^1}<\infty.
\end{aligned}
\ee

By combining \eqref{ua1}--\eqref{ua3}, we obtain
\bes
\begin{aligned}
|(\delta-\rho)(f)|\lesssim & \sum_{j\le 0}\left|(\delta-\rho)(f_j^\sharp)\right|\lesssim  \|f_0^\sharp\|_{L^\infty}+\sum_{j<0}\|(\delta-\rho)\ast\psi_{2^{-j}}\|_{L^1}\|f_j^\sharp\|_{L^\infty}\\
\le & \|f_0^\sharp\|_{L^\infty}+\sup_{j<0}2^{sj}\|f_j^\sharp\|_{L^\infty}\sum_{j<0}2^{-sj}\|(\delta-\rho)\ast\psi_{2^{-j}}\|_{L^1}\lesssim \|f\|_{X^s_{\infty,\infty}},
\end{aligned}
\ees
and thus \eqref{vw1} holds.

\medskip
\noindent
\enquote{\eqref{vw1}$\implies$\eqref{ac1}}. We start by noting that an equivalent formulation of   \eqref{vw1} is
\be
\l{hi1}
\left[f=\sum_{j\in J}f_j^\sharp,\ \text{with }f_j^\sharp\text{ as in }\eqref{vv1}\text{ and }J\subset\Z_-\text{ finite}\right]\implies \left|(\delta-\rho)\left(\sum_{j\in J}f_j^\sharp\right)\right|\lesssim \sup_{j\in J}2^{sj}\|f_j^\sharp\|_{L^\infty}.
\ee

Step 1 in the proof of Theorem \ref{ab2} implies that, if we find some $\lambda\in\sr$ such that $\int\lambda\neq 0$ and 
\be
\l{ub1}
\sum_{j\ge 0}2^{sj}\|\lambda-\lambda\ast\rho_{2^{-j}}\|_{L^1}<\infty,
\ee
then \eqref{ac1} holds. 

Let $\zeta$, $\va$ be as in the Littlewood-Paley decomposition. We will prove that \eqref{ub1} holds with $\lambda=\zeta$. 

Set 
\bes
\alpha_j:=\|\va_{2^{j}}-\va_{2^{j}}\ast\rho\|_{L^1}=\|(\va-\va\ast\rho_{2^{-j}})_{2^{j}}\|_{L^1}=\|(\va-\va\ast\rho_{2^{-j}})\|_{L^1},\ \fo j>0.
\ees

We divide the proof of \eqref{ub1} into two steps.

\medskip
\noindent
{\it Step 1.}  It suffices to prove the key estimate
\be
\l{ub2}
\sum_{j>0}2^{sj}\alpha_j<\infty.
\ee

Granted \eqref{ub2}, we prove \eqref{ub1} for $\lambda=\zeta$. Indeed, using the fact that 
\bes
\lim_{M\to\infty}\|\zeta_M-\zeta_M\ast\rho\|_{L^1}=\lim_{M\to\infty}\|\zeta-\zeta\ast\rho_{1/M}\|_{L^1}=0,
\ees
we find that, in $L^1$, we have
\be
\l{ub3}
\lim_{\ell\to\infty}\ \sum_{j=k+1}^{\ell} (\va_{2^{j}}-\va_{2^{j}}\ast\rho)=\lim_{\ell\to\infty}\left[ (\zeta_{2^k}-\zeta_{2^k}\ast\rho)-(\zeta_{2^\ell}-\zeta_{2^\ell}\ast\rho) \right]=\zeta_{2^k}-\zeta_{2^k}\ast\rho.
\ee

By \eqref{ub3}, we have
\be
\l{ub4}
\|\zeta_{2^k}-\zeta_{2^k}\ast\rho\|_{L^1}\le \sum_{j\ge k+1} \alpha_j.
\ee

By combining \eqref{ub2} with \eqref{ub4}, we obtain
\bes
\sum_{k\ge 0}2^{sk}\|\zeta-\zeta\ast\rho_{2^{-k}}\|_{L^1}=\sum_{k\ge 0}2^{sk}\|\zeta_{2^k}-\zeta_{2^k}\ast\rho\|_{L^1}\le \sum_{k\ge 0}\sum_{j\ge k+1}2^{sk}\alpha_j\lesssim \sum_{j>0}2^{sj}\alpha_j<\infty,
\ees
and thus \eqref{ub1} holds.

\medskip
\noindent
{\it Step 2.}
Proof of \eqref{ub2} completed.

For $\ell<0$, let $\psi_\ell\in C^\infty_c(\R^n)$ be such that $|\psi_\ell|\le 1$ and 
\be
\l{hi9}
\int [(\delta-\rho)\ast\va^\ell]\psi_\ell\ge \frac 12\|(\delta-\rho)\ast\va^\ell\|_{L^1}=\frac 12\alpha_{-\ell}.
\ee

Let $J\subset\Z^\ast_-$ be a fixed arbitrary finite set, and set 
\bes
f:=\sum_{\ell\in J}2^{-s\ell}\psi_\ell\ast\va^\ell.
\ees

By \eqref{hi9}, we have (using \eqref{abade})
\be
\l{hi2}
\sum_{\ell\in J}2^{-s\ell}\alpha_{-\ell}\lesssim \sum_{\ell\in J}2^{-s\ell}\int \left[(\delta-\rho)\ast\va^\ell\right]\, \psi_\ell=(\delta-\rho)\left(\sum_{\ell\in J}2^{-s\ell}\psi_\ell\ast\va^\ell\right).
\ee
By \eqref{hi1} and \eqref{hi2}, we have
\be
\l{hi3}
\sum_{\ell\in J}2^{-s\ell}\alpha_{-\ell}\lesssim \sup_{j\in M}2^{sj}\|f_j^\sharp\|_{L^\infty},
\ee
where $M\subset\Z_-$ is finite and such that $f_j^\sharp=0$ when $j\not\in M$.\footnote{Existence of such $M$ follows from \eqref{hi4} below.}

We next note that, when $j,\ell<0$, we have
\be
\l{hi7}
\va^\ell\ast\va^j=0\text{ when }|j-\ell|>1.
\ee

By \eqref{hi7}, when $j<0$ we have
\be
\l{hi4}
f_j^\sharp=\sum_{\ell\in J}2^{-s\ell}\left(\psi_\ell\ast\va^\ell\right)^\sharp_j=\sum_{\ell\in J}2^{-s\ell}\psi_\ell\ast\va^\ell\ast\va^j=\sum_{\substack{\ell\in J\\|\ell-j|\le 1}}2^{-s\ell}\psi_\ell\ast\va^\ell\ast\va^j,
\ee
and thus
\be
\l{hi5}
\|f_j^\sharp\|_{L^\infty}\lesssim \sum_{\substack{\ell\in J\\|\ell-j|\le 1}}2^{-s\ell}\|\psi_\ell\|_{L^\infty}\lesssim 2^{-sj}.
\ee

By \eqref{hi3} and \eqref{hi5}, we have
\be
\l{hi6}
\sum_{\ell\in J}2^{-s\ell}\alpha_{-\ell}\le C<\infty,
\ee
with $C$ independent of $J$.

We obtain \eqref{ub2} by taking, in \eqref{hi6}, the supremum over $J$.
\ep

\section{Further results}
\l{s5}

This section is devoted to the proofs of Propositions \ref{ab5}, \ref{ab7} and \ref{ab8}.

\subsection{Proof of Proposition~\ref{ab5}}

Proposition~\ref{ab5} is a direct consequence of the following more general result.

\bpr
\label{pmoments}
Let $\rho\in L^1$ satisfy $\int\rho=1$ and let $\eta\in\sr$ be such that $\int \eta \neq 0$.
Assume that $\rho$ has finite moments of any order:
\begin{equation*}
\int |y|^{k}|\rho(y)|\, dy <\infty \quad\text{for all } k\in\N.
\end{equation*}

Then
\begin{equation}\label{charmoments}
\int_0^1\frac 1{\ve^{s+1}}\|\eta-\eta\ast\rho_\ve\|_{L^1}\, d\ve<\infty \quad\text{if and only if}\quad s<k_0,
\end{equation}
where $k_0\in\N^*\cup \lbrace \infty \rbrace$ is the smallest non-zero moment of $\rho$:
\begin{equation*}
k_0 =\min \left\lbrace k\geq 1\colon \int y^{\otimes k} \rho(y)\, dy \neq 0\right\rbrace.
\end{equation*}
Here $y^{\otimes k}$ denotes the $k$-th order tensor $(y_{j_1}\cdots y_{j_k})_{1\leq j_1,\ldots,j_k\leq n}$.
\epr

Note that Proposition~\ref{pmoments} implies indeed Proposition~\ref{ab5} since for a bounded set $A$ of positive measure the second moment $\int_A y^{\otimes 2}\, dy$ is always non zero.

We now turn to the 
\bp[Proof of Proposition~\ref{pmoments}] We first treat  the case of a finite $k_0$. Since it holds
\begin{equation*}
\eta(x)-\eta\ast\rho_\ve (x) =\int (\eta(x)-\eta(x-\varepsilon y))\rho(y)\, dy,
\end{equation*}
we find, applying Taylor's formula,
\begin{equation*}
\eta(x)-\eta\ast\rho_\ve (x) = \frac{(-1)^{k_0+1}}{k_0 !}\varepsilon ^{k_0}\sum_{1\leq j_1,\ldots,j_{k_0}\leq n}\alpha_{j_1,\ldots,j_{k_0}} \partial_{j_1}\cdots\partial_{j_{k_0}}\eta(x) + \varepsilon^{k_0+1}R_\varepsilon (x),
\end{equation*}
where
\begin{equation}\label{alpha}
\alpha_{j_1,\ldots,j_{k}} := \int y_{j_1}\cdots y_{j_{k}}\rho(y)\, dy,
\end{equation}
and
\begin{equation*}
\| R_\varepsilon \|_{L^1}\leq \frac{ \|D^{k_0+1}\eta\|_{L^1}}{(k_0+1)!}\int |y|^{k_0+1}|\rho(y)|\, dy.
\end{equation*}

Therefore it holds
\begin{equation}\label{taylor}
\|\eta-\eta\ast\rho_\varepsilon \|_{L^1} = \frac{1}{k_0!}\varepsilon^{k_0} \Big\|\sum_{1\leq j_1,\ldots,j_{k_0}\leq n}\alpha_{j_1,\ldots,j_{k_0}} \partial_{j_1}\cdots\partial_{j_{k_0}}\eta \Big\|_{L^1} + O(\varepsilon^{k_0+1}),
\end{equation}
as $\varepsilon\to 0$.

We next claim that 
\begin{equation*}
c:=\Big\|\sum_{1\leq j_1,\ldots,j_{k_0}\leq n}\alpha_{j_1,\ldots,j_{k_0}} \partial_{j_1}\cdots\partial_{j_{k_0}}\eta \Big\|_{L^1} \neq 0.
\end{equation*}

Indeed, assume that $c=0$. Then we have
\begin{equation*}
\sum_{1\leq j_1,\ldots,j_{k_0}\leq n}\alpha_{j_1,\ldots,j_{k_0}} \xi_{j_1}\cdots \xi_{j_{k_0}} \hat\eta(\xi) =0 \quad\forall\xi\in\R^n.
\end{equation*}
Since $\hat\eta(0)\neq 0$ we deduce that
\begin{equation*}
\sum_{1\leq j_1,\ldots,j_{k_0}\leq n}\alpha_{j_1,\ldots,j_{k_0}} \xi_{j_1}\cdots \xi_{j_{k_0}} =0
\end{equation*}
for all sufficiently small $\xi$, and thus by homogeneity for every $\xi$. This is absurd since, by assumption, at least one of the coefficients $\alpha_{j_1,\ldots,j_{k_0}}$ is non zero.

Therefore $c\neq 0$ and the Taylor expansion \eqref{taylor} provides the equivalent
\begin{equation*}
\|\eta-\eta\ast\rho_\varepsilon \|_{L^1} \sim \frac{c}{k_0!} \varepsilon^{k_0}
\end{equation*}
as $\varepsilon\to 0$, which readily implies \eqref{charmoments}. This concludes the proof of Proposition~\ref{pmoments} when $k_0$ is finite. 

When $k_0=\infty$, the Taylor expansion shows that 
\begin{equation*}
\|\eta-\eta\ast\rho_\ve\|_{L^1} = O(\varepsilon ^k) \quad\text{for all }k\in\N,
\end{equation*}
so that it holds indeed
\begin{equation*}
\int_0^1\frac 1{\ve^{s+1}}\|\eta-\eta\ast\rho_\ve\|_{L^1}\, d\ve<\infty
\end{equation*}
for every $s>0$.
\ep

\subsection{Proof of Proposition~\ref{ab7}}

We fix $\rho\in L^1$ with $\int\rho=1$ and $0<s<1$, and assume that $\rho$ satisfies the moment condition \eqref{moment}:
\begin{equation*}
\int |y|^s |\rho(y)|\, dy < \infty.
\end{equation*}

We consider an arbitrary test function $\eta\in\sr$ and are going to show that condition \eqref{ac1} is satisfied (so that, by Theorem~\ref{ab2}, the norm equivalence \eqref{ab3} is valid). To this end we compute
\begin{align*}
\int_0^\infty \| \eta-\eta\ast\rho_\ve\|_{L^1}\,  \frac{d\varepsilon}{\varepsilon^{s+1}} & \leq \int_0^\infty \int \|\eta-\eta(\cdot-\ve y)\|_{L^1} |\rho(y)|\,  dy \frac{d{\ve}}{{\ve^{s+1}}} \\
& = \int |y|^s\rho(y) \int_0^\infty \frac{\|\eta-\eta(\cdot-\varepsilon y)\|_{L^1}}{|\varepsilon y|^s}\, \frac{d\ve}{\ve}\, dy \\
& = \int |y|^s\rho(y) \int_0^\infty \|\eta-\eta(\cdot-\delta \frac{y}{|y|})\|_{L^1}\, \frac{d\delta}{\delta^{s+1}}\, dy.
\end{align*}

On the other hand, for every $\omega\in \mathbb S^{n-1}$ we have the estimate
\begin{equation*}
\int_0^\infty \|\eta-\eta(\cdot-\delta \omega)\|_{L^1}\, \frac{d\delta}{\delta^{s+1}}
\leq \|D\eta\|_{L^1}\int_0^1\, \frac{d\delta}{\delta^s} + 2\|\eta\|_{L^1}\int_1^\infty\,  \frac{d\delta}{\delta^{s+1}}=:C(\eta)<\infty,
\end{equation*}
and therefore we conclude that
\begin{equation*}
\int_0^\infty \| \eta-\eta\ast\rho_\ve\|_{L^1}\, \frac{d\varepsilon}{\varepsilon^{s+1}} \leq C(\eta) \int |y|^s |\rho(y)|\, dy < \infty,
\end{equation*}
which finishes the proof of Proposition~\ref{ab7}.\hfill$\square$

\subsection{Proof of Proposition~\ref{ab8}}

Let $s>0$ and let $\rho\in L^1$ satisfy $\int\rho =1$ and $\rho\geq 0$. We assume that the norm equivalence \eqref{ab3} is valid. Then by Theorem~\ref{ab2} (and Step 1 in its proof), it holds
\begin{equation*}
\int_0^1 \| \eta-\eta\ast\rho_\ve\|_{L^1} \, \frac{d\varepsilon}{\varepsilon^{s+1}}<\infty
\end{equation*}
for every $\eta\in\sr$. We fix such a function $\eta\geq 0$, $\eta\not\equiv 0$, with support in the unit ball:
\begin{equation*}
\eta (x) = 0 \quad\text{for }|x|\geq 1.
\end{equation*}
We are going to show that
\begin{equation}\label{necmom}
 \int_0^1 \|\eta-\eta\ast\rho_\varepsilon\|_{L^1}\, \frac{d\ve}{\varepsilon^{s+1}}  \geq c \|\eta\|_{L^1}\int |y|^s \rho(y)\, dy -C (\|\eta\|_{L^1}+\|\eta\|_{L^\infty} \|\rho\|_{L^1}),
\end{equation}
for some  constants $c=c(s),C=C(s)>0$. Obviously \eqref{necmom} implies the conclusion of Proposition~\ref{ab8}: the function $\rho$ satisfies the finite moment condition
\begin{equation*}
\int |y|^s\rho(y)\, dy <\infty.
\end{equation*}

We now turn to  the proof of \eqref{necmom}. Note that 
\begin{equation*}
\int_1^\infty \frac{1}{\varepsilon^{s+1}}\|\eta-\eta\ast\rho_\varepsilon\|_{L^1}\, d\varepsilon \leq \int_1^\infty\, \frac{d\varepsilon}{\varepsilon^{s+1}}(\|\eta\|_{L^1}+ \|\eta\|_{L^\infty}\|\rho\|_{L^1}).
\end{equation*}
Hence it suffices to show that
\begin{equation*}
\int_0^\infty \frac{1}{\varepsilon^{s+1}}\|\eta-\eta\ast\rho_\varepsilon\|_{L^1}\, d\varepsilon \geq c \|\eta\|_{L^1}\int |y|^s \rho(y)\, dy.
\end{equation*}
Since $\eta(x)=0$ for $|x|\geq 1$, and since $\eta$ and $\rho$ are non negative, it holds
\begin{align*}
\| \eta-\eta\ast\rho_\varepsilon \|_{L^1} & \geq \iint_{|x|\geq 1} \eta(x-\varepsilon y)\rho(y)\,  dy  = \iint_{|z+\varepsilon y|\geq 1} \eta(z)\rho(y)\,  dy dz.
\end{align*}
Thus we obtain
\be
\l{ana1}
\begin{aligned}
\int_0^\infty \frac{1}{\varepsilon^{s+1}}\|\eta-\eta\ast\rho_\varepsilon\|_{L^1}\, d\varepsilon & 
\geq \iiint_{|z+\varepsilon y|\geq 1} \eta(z)\frac{\rho(y)}{\varepsilon^{s+1}}\, dy dz d\varepsilon \\
& = \iiint_{|z+\delta {y}/{|y|}|\geq 1} \eta(z)\frac{\rho(y)|y|^s}{\delta^{s+1}}\, dy dz d\delta.
\end{aligned}
\ee
Note that it holds
\begin{equation*}
[|\delta|\ge 2\text{ and }|z|<1]\implies |z+\delta {y}/{|y|}|\geq 1.
\end{equation*}
Therefore, the domain of integration in the last integral in \eqref{ana1} contains the set
\bes
\{ (y,z,\delta);\, y\neq 0,\, |z|<1,\, \delta\ge 2\}.
\ees
We find that
\begin{equation*}
\int_0^\infty \frac{1}{\varepsilon^{s+1}}\|\eta-\eta\ast\rho_\varepsilon\|_{L^1}\, d\varepsilon 
 \geq \|\eta\|_{L^1} \int_2^\infty\, \frac{d\delta}{\delta^{s+1}}\int \rho(y)|y|^s\, dy,
\end{equation*}
which completes the proof of \eqref{necmom}.\hfill$\square$

\bibliographystyle{plain}               
  \bibliography{bibsobolev.bib}   

\end{document}